\documentclass[11pt,fullpage,doublespace]{article}
\usepackage{graphicx} 
\usepackage{setspace} 
\usepackage{latexsym} 
\usepackage{amsfonts} 
\usepackage{amsmath} 
\usepackage{amssymb}
\usepackage{accents}
\usepackage{textcomp}
\usepackage{enumerate}
\usepackage{mathrsfs}
\usepackage{bm}
\usepackage{stmaryrd}
\usepackage[margin=1in]{geometry}
\usepackage{amsthm}
\usepackage{ifsym}
\usepackage{amssymb,latexsym,amsmath}
\usepackage{graphics}
\usepackage{thmtools}
\usepackage{comment}

\usepackage[colorlinks=true,linkcolor=blue]{hyperref}
\declaretheorem[style=definition,qed=$\dashv$,numberwithin=section]
{definition}

\declaretheorem[style=plain,sibling=definition]{theorem}
\declaretheorem[style=plain,sibling=definition]{lemma}

\declaretheorem[style=plain,sibling=definition]{question}
\declaretheorem[style=plain,sibling=definition]{corollary}
\declaretheorem[style=definition,sibling=definition]{remark}
\declaretheorem[style=definition,sibling=definition]{conjecture}
\declaretheorem[style=plain,sibling=definition]{claim}
\declaretheorem[style=plain,sibling=definition]{claim*}
\numberwithin{equation}{section}

\usepackage{titlesec}
\titleformat{\section}{\normalsize\centering}{\thesection.}{1em}{}
\titleformat{\subsection}{\normalsize\centering}{\thesubsection.}{1em}{}
\numberwithin{equation}{section}

\newcommand{\rest}{\restriction}

\newcommand{\powerset}{{\wp }}

\def\P{{\mathcal{P} }}

\def\Q{{\mathcal{ Q}}}

\def\M{{\mathcal{M}}}

\def\T {{\mathcal{T}}}

\newcommand*{\TitleFont}{%
      \usefont{\encodingdefault}{\rmdefault}{b}{n}%
      \fontsize{12}{16}%
      \selectfont}

 \input xy
 \xyoption{all}
\begin{document}
\title{\TitleFont ON A CLASS OF  MAXIMALITY PRINCIPLES}
\renewcommand{\thefootnote}{\fnsymbol{footnote}} 
\renewcommand{\thefootnote}{\arabic{footnote}}
\author{\fontsize{11}{13} \selectfont DAISUKE IKEGAMI\footnote{Department of Mathematics, School of Engineering, Tokyo Denki University. E-mail: ikegami@mail.dendai.ac.jp} \\ \fontsize{11}{13} \selectfont NAM TRANG\footnote{Department of Mathematics, UC Irvine, CA, USA. Email: ntrang@math.uci.edu}}

\date{}
\maketitle
\begin{abstract}
We study various classes of maximality principles, $\rm{MP}(\kappa,\Gamma)$, introduced by J.D. Hamkins in \cite{hamkins2000simple}, where $\Gamma$ defines a class of forcing posets and $\kappa$ is an infinite cardinal. We explore the consistency strength and the relationship of $\textsf{MP}(\kappa,\Gamma)$ with various forcing axioms when $\kappa \in\{\omega,\omega_1\}$. In particular, we give a characterization of bounded forcing axioms for a class of forcings $\Gamma$ in terms of maximality principles MP$(\omega_1,\Gamma)$ for $\Sigma_1$ formulas. A significant part of the paper is devoted to studying the principle MP$(\kappa,\Gamma)$ where $\kappa\in\{\omega,\omega_1\}$ and $\Gamma$ defines the class of stationary set preserving forcings. We show that MP$(\kappa,\Gamma)$ has high consistency strength; on the other hand, if $\Gamma$ defines the class of proper forcings or semi-proper forcings, then by \cite{hamkins2000simple}, MP$(\kappa,\Gamma)$ is consistent relative to $V=L$.
\end{abstract}

\section{INTRODUCTION}
\begin{definition}
Let $\Gamma$ be a formula which defines a class of forcing posets (e.g. ccc, proper etc) and $\phi$ be a formula (in the language of set theory). We say that $\phi$ is \textbf{$\Gamma$-necessary} (over $V$) if $\phi$ is true in $V^{\mathbb{P}}$ for any $\mathbb{P}$ such that $\Gamma(\mathbb{P})$ holds (informally, we just write $\mathbb{P}\in \Gamma$). We say that $\phi$ is \textbf{forcibly $\Gamma$-necessary} if there is some $\mathbb{P}\in \Gamma$ such that $\phi$ is $\Gamma$-necessary over $V^{\mathbb{P}}$, i.e. $V^\mathbb{P} \vDash ``\phi$ is $\Gamma$-necessary".
\end{definition}

The following definition comes from \cite{hamkins2000simple}.

\begin{definition}
Let $\kappa$ be an infinite cardinal and $\Gamma$ define a class of forcing posets. The principle $\rm{MP}(\kappa,\Gamma)$ states that: for any $A\subseteq \kappa$ and any formula $\phi(v)$, if there is a $\mathbb{P}\in \Gamma$ such that $\phi(A)$ is $\Gamma$-necessary over $V^{\mathbb{P}}$, then $\phi(A)$ holds in $V$. The principle $\rm{NMP}$$(\kappa,\Gamma)$ states that  $\rm{MP}(\kappa,\Gamma)$ holds in $V^\mathbb{P}$ for all $\mathbb{P}\in \Gamma$. 
\end{definition}

$\rm{MP}$$(\kappa,\Gamma)$ and $\rm{NMP}$$(\kappa,\Gamma)$ are generally axiom schemata and one can formulate them in $\sf{ZFC}$. See \cite{hamkins2000simple} for a more detailed discussion.

Clearly, $\rm{NMP}(\kappa,\Gamma)$ implies $\rm{MP}(\kappa,\Gamma)$; also, if $\kappa<\lambda$, then MP$(\lambda,\Gamma)$ implies MP$(\kappa,\Gamma)$. On the other hand, it's not clear if there is any relationship at all between $\rm{MP}(\kappa,\Gamma_1)$ ($\rm{NMP}(\kappa,\Gamma_1)$, respectively) and $\rm{MP}(\kappa,\Gamma_2)$ ($\rm{NMP}(\kappa,\Gamma_2)$, respectively) even if $\Gamma_1 \subseteq \Gamma_2$.\footnote{Here, we identify the formulas $\Gamma_1$, $\Gamma_2$ with the classes of forcings they define. $\Gamma_1\subseteq \Gamma_2$ means $\forall x (\Gamma_1(x) \Rightarrow \Gamma_2(x))$. We will make this identification throughout the paper.}

Woodin (unpublished) has shown that if $\Gamma$ defines the class of all forcing posets, the principle $\rm{NMP}(\omega,\Gamma)$ is consistent relative to ``$\sf{AD}_\mathbb{R}+\Theta  $ is regular". Hamkins and Woodin (cf. \cite{hamkins2005necessary}) have shown that if $\Gamma$ is the class of ccc forcings, then $\rm{NMP}(\omega,\Gamma)$ is equiconsistent with ``$\sf{ZFC} + $ there is a weakly compact cardinal." For $\Gamma$ being the class of  proper, or semi-proper, or stationary set preserving forcings, we believe it's open whether $\rm{NMP}(\omega,\Gamma)$ is consistent. 

As remarked in \cite[page 22]{hamkins2000simple}, for any sufficiently rich $\Gamma$ and any $\kappa > \omega$, the principle $\rm{NMP}$$(\kappa,\Gamma)$ in general will be false (see, for example, Corollary \ref{boldfaceNMPfails}). However, $\rm{MP}(\omega_1,\Gamma)$ may still hold for some $\Gamma$.

In this paper, we explore the consistency strength of various maximality principles (as defined above) for several important classes of forcing posets, and their relations with various well-known forcing axioms such as The Martin Maximum ($\sf{MM}$), Woodin's $(*)$ axiom, and bounded forcing axioms.

\begin{definition}
\begin{enumerate}
\item A cardinal $\lambda$ is \textbf{reflecting} if $V_\lambda \prec V$.\footnote{This type of cardinals can be formalized in $\sf{ZFC}$ just by enriching the language with a parameter $\lambda$ for the relevant cardinal and expressing the elementarity of $V_\lambda$ with $V$ by means of infintely many formulae in parameter $\lambda$.}
\item A cardinal $\kappa$ is \textbf{hyper-huge} if whenever $\lambda>\kappa$, there is an elementary embedding $j:V\rightarrow M$ such that crt$(j)=\kappa$, $j(\kappa)>\lambda$, and $M^{j(\lambda)}\subset M$.
\end{enumerate}
\end{definition}
\cite{hamkins2000simple} shows that MP$(\omega_1,\Gamma)$ has very low consistency strength for $\Gamma$ defining the class of $\sigma$-closed, proper, or semi-proper forcings; in particular, \cite{hamkins2000simple} shows that these principles are consistent relative to $V=L$. On the other hand, if $\Gamma$ defines the class of stationary set preserving forcings, Theorem \ref{BMM} shows MP$(\omega,\Gamma)$ implies that $0^\sharp$ exists and more. If $\Gamma$ is the class of stationary set preserving forcings, $\rm{MP}$$(\omega,\Gamma)$ and $\rm{MP}$$(\omega_1,\Gamma)$ may be very strong; but  $\rm{MP}$$(0,\Gamma)$, on the other hand, is consistent relative to $\sf{ZFC}$ (by \cite{hamkins2000simple}).

The following theorem deals with NMP for the class of $\sigma$-closed forcings.

\begin{theorem}\label{FromReflectingCard}
\begin{enumerate}[(1)]
\item Suppose $\Gamma$ defines the class of $\sigma$-closed forcings. Then $\rm{NMP}(\omega,\Gamma)$ is consistent relative to $\sf{ZFC}$.
\item Suppose $\Gamma$ defines the class of $\sigma$-closed forcings. Suppose there is a hyper-huge cardinal. Then $\rm{NMP}$$(\omega_1,\Gamma)$ fails.
\end{enumerate}
\end{theorem}

\begin{remark}
We do not know if NMP$(\omega_1,\Gamma)$ is consistent, where $\Gamma$ defines the class of $\sigma$-closed forcings.
\end{remark}

The next theorem gives an upper bound consistency strength for MP$(\omega_1,\Gamma)$, where $\Gamma$ defines the class of stationary set preserving forcings.

\begin{theorem}
\label{FromSupercompact}

Let $\Gamma$ define the class of stationary set preserving forcing posets. Suppose there is a proper class of strongly compact cardinals and an inaccessible cardinal which is reflecting. Then in some generic extension of $V$, $\rm{MP}(\omega_1,\Gamma)$ holds.
\end{theorem}

The proofs of Theorems \ref{FromReflectingCard} and \ref{FromSupercompact} form the content of Section \ref{UpperBounds}.

\begin{definition}\label{StratifiedMP}
Let $\Gamma$ define a class of forcing posets, $\kappa$ is a cardinal. For each $n\leq \omega$, $\rm{MP}$$_{\Sigma_n}(\kappa,\Gamma)$ is the restriction of $\rm{MP}$$(\kappa,\Gamma)$ to $\Sigma_n$ formulas. More precisely, $\rm{MP}$$_{\Sigma_n}(\kappa,\Gamma)$ is the statement: suppose $\phi(v)$ is a $\Sigma_n$ formula (in the language of set theory), $A\subseteq \kappa$, and $\mathbb{P}\in\Gamma$ is such that $\phi(A)$ is $\Gamma$-necessary over $V^{\mathbb{P}}$, then $\phi(A)$ is true.

We define $\rm{MP}$$_{\Pi_n}(\kappa,\Gamma)$, $\rm{NMP}$$_{\Sigma_n}(\kappa,\Gamma)$, $\rm{NMP}$$_{\Pi_n}(\kappa,\Gamma)$ etc. similarly.
\end{definition}

The following two theorems establish some connections between the forcing axioms $\sf{MM}^{++}$, $\rm{MP} (\omega, \Gamma)$ and $\rm{MP} (\omega_1, \Gamma)$ where $\Gamma$ is the class of stationary set preserving forcings. 
\begin{theorem}\label{MM++vsMP}
Let $\Gamma$ define the class of all stationary set preserving forcings. 
\begin{enumerate}
\item Suppose that $\sf{MM}^{++}$ holds and that there are proper class many Woodin cardinals. Then $\rm{MP}$$_{\Pi_2}(\omega_1, \Gamma)$ holds.

\item ``$\sf{MM}^{++}$ + there are proper class many Woodin cardinals" does not imply $\rm{MP}$$(\omega_1, \Gamma)$.

\item $\sf{MM}^{++}$ does not imply $\rm{MP}$$_{\Pi_2}(\omega_1, \Gamma)$.
\end{enumerate}
\end{theorem}

\begin{theorem}\label{NothingImpliesAnything}
Let $\Gamma$ define the class of stationary set preserving forcing posets.
\begin{enumerate}
\item $\rm{MP}$$(\omega_1,\Gamma)$ does not imply $\sf{MM}^{++}$.
\item $\rm{MP}$$(\omega_1,\Gamma)$ implies $\sf{BMM}$ and $\sf{BMM}$ does not imply $\rm{MP}$$(\omega_1,\Gamma)$.
\item $\rm{MP}$$(\omega,\Gamma)$ does not imply $\rm{MP}$$(\omega_1,\Gamma)$ and $\rm{MP}$$(0,\Gamma)$ does not imply $\rm{MP}$$(\omega,\Gamma)$.
\end{enumerate}

\end{theorem}

We now discuss the connection between $\text{MP}(\omega_1, \Gamma)$ and Woodin's Axiom (*), where $\Gamma$ is the class of stationary set preserving forcings. 

\begin{theorem}\label{MM++vsStar}
Let $\Gamma$ define the class of stationary set preserving forcings. 
\begin{enumerate}
\item Woodin's Axiom (*) does not imply $\rm{MP}$$(\omega_1, \Gamma)$. 

\item $\rm{MP}$$(\omega_1, \Gamma)$ does not imply Woodin's Axiom (*).
\end{enumerate}
\end{theorem}

\begin{definition}
\label{BoundedForcingAxioms}

Let $\Gamma$ define a class of forcing posets and $\kappa$ be an uncountable cardinal. $\sf{FA}_\kappa(\Gamma)$ is the following statement: for any $\mathbb{P}\in \Gamma$, let $\langle D_\alpha \ | \ \alpha<\omega_1 \rangle$ be a sequence of maximal antichains of r.o.$(\mathbb{P})$\footnote{r.o.$(\mathbb{P})$ is the completion of $\mathbb{P}$ and is a complete Boolean algebra.} such that for each $\alpha<\omega_1$, $|D_\alpha| \leq \kappa$, then there is a filter $G\subseteq$ r.o.$(\mathbb{P})$ such that $G\cap D_\alpha\neq \emptyset$ for all $\alpha<\omega_1$. \end{definition}

\begin{remark}
If $\kappa=\aleph_1$, then by \cite{Bagaria2000}, $\sf{FA}_\kappa(\Gamma)$ is equivalent to the statement: for any $\mathbb{P}\in \Gamma$, 
\begin{center}
$(H_{\omega_2},\in) \prec_{\Sigma_1} (H_{\omega_2}^{V^\mathbb{P}},\in)$.
\end{center}

\end{remark}

The following gives a characterization of bounded forcing axioms of the form $\sf{FA}_{\aleph_1}(\Gamma)$ in terms of maximality principles for $\Sigma_1$ statements.
\begin{theorem}\label{MPvsFA}
Let $\Gamma$ define a class of complete Boolean algebras. Then $\sf{FA}_{\aleph_1}(\Gamma)$ is equivalent to $\rm{MP}$$_{\Sigma_1}(\omega_1,\Gamma)$.
\end{theorem}

From Theorem \ref{MPvsFA}, we get the following.
\begin{corollary}
\label{boldfaceNMPfails}
Suppose $\Gamma$ defines the class of proper forcings, semi-proper forcings, or stationary set preserving posets. Then $\rm{NMP}$$(\omega_1,\Gamma)$ is false.
\end{corollary}
\begin{proof}
Let $\Gamma$ define the class of proper forcings (the proof is the same for the other classes). By standard results (cf. \cite{moore2005set}), $\sf{FA}$$_{\aleph_1}(\Gamma)$ (commonly known as $\sf{BPFA}$) implies $\sf{CH}$ fails. Let $\mathbb{P}$ be the ($\sigma$-closed, hence proper) forcing that adds a Cohen subset of $\omega_1$. Then in $V^{\mathbb{P}}$, $\sf{CH}$ holds and NMP$(\omega_1,\Gamma)$ holds. In particular, by Theorem \ref{MPvsFA}, $\sf{FA}_{\aleph_1}(\Gamma)$ holds in $V^\mathbb{P}$. But then $\sf{CH}$ fails in $V^\mathbb{P}$. Contradiction. 
\end{proof}

We remark that if $\kappa\geq \omega_2$ then $\rm{MP}(\kappa,\Gamma)$ may fail. For instance, let $\Gamma$ define the class of proper forcings and $\kappa=\aleph_2 = 2^{\aleph_0}$ (by Theorem \ref{MPvsFA} and \cite{moore2005set}). Let $A\subseteq \kappa$ code the reals and consider the statement $\phi(A) \equiv$ ``there is a real $x\notin A$". Obviously, if $\mathbb{P}$ is the Cohen forcing that adds a Cohen real, then $\phi(A)$ can be made $\Gamma$-necessary over $V^\mathbb{P}$ but cannot be true in $V$ by the definition of $A$.

In Section \ref{Relationships}, we prove Theorems \ref{MM++vsMP}, \ref{NothingImpliesAnything}, \ref{MM++vsStar}, and \ref{MPvsFA}.

The following theorem suggests that when $\Gamma$ defines the class of stationary set preserving posets, the principles $\rm{MP}(\omega,\Gamma)$ and $\rm{MP}(\omega_1,\Gamma)$ may have considerable consistency strength.
\begin{theorem}
\label{BMM}
If $\Gamma$ is the class of stationary set preserving forcings, then $\rm{MP}(\omega_1,\Gamma)$ implies $2^{\aleph_0}= 2^{\aleph_1}=\aleph_2$ and for all $X$, $X^\sharp$ exists. In fact, the second clause follows from $\rm{MP}(\omega,\Gamma)$. Furthermore, $\sf{MM}(c)$ does not imply $\rm{MP}(\omega_1,\Gamma)$.
\end{theorem}

When combined with other mild assumptions, $\rm{MP}(\kappa,\Gamma)$ can have significant lower-bound consistency strength for various classes of forcings $\Gamma$. For instance, we have the following.
\begin{theorem}
\label{gettingPD}
Suppose $\Gamma$ defines the class of $\sigma$-closed, proper forcing posets, semi-proper forcing posets, or stationary set preserving forcing posets. Suppose $\rm{MP}(\omega_1,\Gamma)$ holds and there is a precipitous ideal on $\omega_1$. Then Projective Determinacy $(\sf{PD})$ holds.
\end{theorem}

The proofs of Theorems \ref{BMM} and \ref{gettingPD} will be given in Section \ref{Lowerbounds}. In Section \ref{OpenProblems}, we list some related open problems and questions.

There have been several recent results on generic absoluteness closely related to this work. The reader can see, for instance, \cite{MR3194674}, \cite{MR3377357}, and \cite{MR3486170}.

\textbf{Acknowledgement.} The paper was written during the second author's visit to Kobe University, where the first author was a postdoctoral researcher in May 2014 and completed during the first author's visit to UC Irvine, where the second author is a Visiting Assistant Professor, in August 2016. The first author would like to thank Toshimichi Usuba for many comments and discussions on this topic. He is also grateful for JSPS for support through the grants with JSPS KAKENHI Grant Number 14J02269 and 15K17586. The second author would like to thank the NSF for its generous support through grant DMS-1565808. Finally, we would like to thank the anonymous referee for several helpful comments regarding the content of the paper.

\section{UPPER-BOUND CONSISTENCY STRENGTH}\label{UpperBounds}
In this section, we prove Theorems \ref{FromReflectingCard} and \ref{FromSupercompact}. We start with the proof of Theorem \ref{FromReflectingCard}.
\begin{proof}[Proof of Theorem \ref{FromReflectingCard}]

We now prove (1). We let $\mathbb{P} = \langle \mathbb{P}_\alpha, \mathbb{Q}_\beta\ | \ \alpha\leq 2^\omega \wedge \beta<2^\omega \rangle$ be a countable support iteration of $\sigma$-closed forcings defined as follows. Since $\sigma$-closed forcings don't add reals, we let $\langle (\phi_\alpha,x_\alpha) \ | \  \alpha < 2^\omega\rangle$ enumerate (with unbounded repetition) all pairs $(\phi,x)$, where 
\begin{enumerate}[(i)]
\item $x\in \mathbb{R}$, and
\item $\phi$ is a sentence in the forcing language of $\mathbb{P}$ with parameter $x$.
\end{enumerate} 
By induction, for each $\alpha<2^\omega$, if $(\varphi_\alpha,x_\alpha)$ is such that $\varphi_\alpha(x_\alpha)$ is forcibly $\Gamma$-necessary over $V^{\mathbb{P}_\alpha}$ then we choose $\dot{\mathbb{Q}}_\alpha$ so that $\phi_\alpha(x_\alpha)$ is $\Gamma$-necessary over $V^{\mathbb{P}_\alpha\ast \dot{\mathbb{Q}}_\alpha}$. Note that $\mathbb{P}$ is $\sigma$-closed, and hence $\mathbb{R}^V = \mathbb{R}^{V^{\mathbb{P}}}$.

We claim that $V^\mathbb{P}\vDash \rm{NMP}(\omega,\Gamma)$. So let $G\subseteq \mathbb{P}$ be $V$-generic. Let $\alpha<(2^\omega)^V$ and $\mathbb{Q}\in V[G]$ be such that $\varphi_\alpha(x_\alpha)$ is forcibly $\Gamma$-necessary over $V[G]^{\mathbb{Q}}$. We want to show that $V[G]^{\mathbb{Q}} \vDash \varphi_\alpha(x_\alpha)$. By the assumption, $\varphi_\alpha(x_\alpha)$ is $\Gamma$-forcibly necessary over $V[G\rest \alpha]$. By construction, $\varphi_\alpha(x_\alpha)$ is $\Gamma$-necessary over $V[G\rest(\alpha+1)]$, which in turns implies $\varphi_\alpha(x_\alpha)$ holds $V[G]^{\mathbb{Q}}$ as desired.

Part (2) essentially follows from Usuba \cite{Usuba}. In \cite{Usuba}, assuming in $V$ that there is a hyper-huge cardinal, it is shown that the generic mantle $g\mathbb{M}$ is the mantle $\mathbb{M}$ (hence, $\mathbb{M}$ is generically invariant) and that $\mathbb{M}$ is a ground of $V$.\footnote{Recall the mantle $\mathbb{M}$ is the intersection of all grounds of $V$. The generic mantle $g\mathbb{M}$ is the intersection of all grounds of all set generic extensions of $V$.} 
\begin{claim}
NMP$(\omega_1,\Gamma)$ ($\Gamma$ is the class of $\sigma$-closed forcings) implies that for a proper class of $\alpha$, $\alpha^+ > (\alpha^+)^{g\mathbb{M}}$.
\end{claim}
\begin{proof}
The argument is as follows. Fix an $\alpha\geq \omega_1^V$, $\alpha^\omega\subseteq \alpha$, and let $\kappa=(\alpha^+)^{g\mathbb{M}}$. Suppose for contradiction that $\kappa=\alpha^+$. Let $\mathbb{P}= Coll(\omega_1,\alpha)$, so $\mathbb{P}\in\Gamma$. Let $G\subseteq \mathbb{P}$ be $V$-generic. In $V[G]$, $\kappa=(\alpha^+)^{g\mathbb{M}} = \alpha^+$; this is because by our choice of $\alpha$, $\mathbb{P}$ is $\alpha^+$-cc. Let $A\subset \omega_1$ code $\alpha$ and $\mathbb{Q} = Coll(\omega_1,\kappa)$ and $H\subseteq \mathbb{Q}$ be $V[G]$-generic. Note that $\mathbb{P} * \mathbb{Q}\in\Gamma$ and the statement ``$(\alpha^+)^{g\mathbb{M}}<\omega_2$" is of the form $\phi(A)$ and is $\Gamma$-necessary over $V[G][H]$. By NMP$(\omega_1,\Gamma)$, $\phi(A)$ is true in $V[G]$. This is a contradiction. 
\end{proof}

The conclusion of the claim contradicts the fact that $\mathbb{M}$ is a ground of $V$. This is because $V$ is a set-generic extension of $\mathbb{M} = g\mathbb{M}$; and hence there is a cardinal $\beta$ such that for all $\alpha\geq \beta$, $(\alpha^+)^{g\mathbb{M}} = \alpha^+$.
\end{proof} 

\begin{remark}
In fact, Usuba \cite{Usuba} has shown that letting $\Gamma$ be the class of all forcing posets, if NMP$(\omega,\Gamma)$ holds, then there cannot exist a hyper-huge cardinal. Therefore, Woodin's model of NMP$(\omega,\Gamma)$ cannot accommodate hyper-huge cardinals. The proof of Theorem \ref{FromReflectingCard} is based on Usuba's work in \cite{Usuba} and one can use it to reproduce Usuba's aforementioned result.
\end{remark}

\begin{proof}[Proof of Theorem \ref{FromSupercompact}]

Let $\Gamma$ define the class of all stationary set preserving forcings and $\delta$ be an inaccessible cardinal which is reflecting. Note that $\delta$ is a limit of strongly compact cardinals because there are proper class many strongly compact cardinals and $\delta$ is reflecting. We shall show that there is a semi-proper poset $\mathbb{P}$ of size $\delta$ such that $\text{MP}(\omega_1, \Gamma)$ holds in $V^{\mathbb{P}}$. 

We will construct a revised countable support forcing iteration $(\mathbb{P}_{\alpha} , \dot{\mathbb{Q}}_{\alpha} \mid \alpha \le \delta )$ of semi-proper forcings and a sequence $(x_{\alpha} \in \mathcal{P}(\omega_1)^{V^{\mathbb{P}_{\alpha}}} \mid \alpha < \delta)$ with the following properties:
\begin{enumerate}
\item each $\mathbb{P}_{\alpha}$ is in $V_{\delta}$, 

\item any subset of $\omega_1$ in $V^{\mathbb{P}_{\delta}}$ is of the form $x_{\alpha}$ for some $\alpha < \delta$, 

\item for any $\alpha < \delta$, in $V^{\mathbb{P}_{\alpha}}$, $\dot{\mathbb{Q}}_{\alpha}$ is of the form $\text{Coll}(\omega_1 , < \kappa) * \mathbb{Q}$ where $\kappa$ is a strongly compact cardinal less than $\delta$, and

\item for any formula $\phi$ and any $\alpha < \delta$, if $\phi [x_{\alpha}]$ is $\Gamma$-necessary in $V^{\mathbb{P}_{\alpha} * \text{Coll}(\omega_1 , < \kappa) * \mathbb{Q}'}$ for some stationary set preserving poset $\mathbb{Q}'$ in $V^{\mathbb{P}_{\alpha} * \text{Coll}(\omega_1 , < \kappa)}$ where $\dot{\mathbb{Q}}_{\alpha} = \text{Coll}(\omega_1 , < \kappa) * \mathbb{Q}$ for some $\mathbb{Q}$, then $\phi [x_{\alpha}]$ is $\Gamma$-necessary in $V^{\mathbb{P}_{\alpha + 1}}$.
\end{enumerate}

The item 2. can be organized by a standard book-keeping argument. For item 4., we use the assumption that $\delta$ is reflecting so that the poset $\mathbb{Q}$ like a $\mathbb{Q}'$ in the item 4. can be taken in $V_{\delta}^{V^{\mathbb{P}_{\alpha} * \text{Coll}(\omega_1 , < \kappa)}}$ and that one can set $\dot{\mathbb{Q}}_{\alpha} = \text{Coll}(\omega_1, < \kappa) * \mathbb{Q}$ in $V^{\mathbb{P}_{\alpha}}$. Note that the poset $\mathbb{Q}$ is also semi-proper in $V^{\mathbb{P}_{\alpha} * \text{Coll}(\omega_1 , < \kappa)}$ because $\kappa$ is strongly compact in $V^{\mathbb{P}_{\alpha}}$ and so every stationary set preserving forcing is semi-proper in $V^{\mathbb{P}_{\alpha} * \text{Coll}(\omega_1 , < \kappa)}$ by the result of Shelah in \cite[Chapter~XIII, Section~1]{Shelah}.

Let $\mathbb{P} = \mathbb{P}_{\delta}$. We now argue that $\text{MP}(\omega_1, \Gamma)$ holds in $V^{\mathbb{P}}$. Let $\phi$ be a fomula and $x$ be a subset of $\omega_1$ in $V^{\mathbb{P}}$ such that $\phi [x]$ is $\Gamma$-necessary in $V^{\mathbb{P} * \mathbb{Q}}$ for some stationary set preserving $\mathbb{Q}$ in $V^{\mathbb{P}}$. We will show that $\phi [x]$ holds in $V^{\mathbb{P}}$. By the item 2. above, there is an $\alpha < \delta$ such that $x = x_{\alpha}$ in $V^{\mathbb{P}_{\alpha}}$. Since $\mathbb{P} / \mathbb{P}_{\alpha}$ is semi-proper in $V^{\mathbb{P}_{\alpha}}$, $(\mathbb{P} / \mathbb{P}_{\alpha}) * \mathbb{Q}$ forces that $\phi [x_{\alpha}]$ is $\Gamma$-necessary. So by the item 4., $\phi [x_{\alpha}]$ is $\Gamma$-necessary in $V^{\mathbb{P}_{\alpha +1}}$. Since $\mathbb{P}/\mathbb{P}_{\alpha +1}$ is stationary set preserving in $V^{\mathbb{P}_{\alpha +1}}$, $\phi [x]$ holds in $V^{\mathbb{P}}$, as desired.

\end{proof}

\begin{remark}
The proof above can be modified to show that the axioms $\sf{MM}^{++}$, MP$(\omega_1,\Gamma)$ may be jointly consistent. Assuming even stronger large cardinal properties, a variation of the construction above and \cite[Theorem 3.10]{iterated_resurrection} imply that $\sf{MM}^{++}$, MP$(\omega_1,\Gamma)$, and the iteration resurrection axioms RA$_{\sf{ON}}(\sf{SSP})$ (cf. \cite{iterated_resurrection}) may all be jointly consistent.
\end{remark}

\section{RELATIONSHIP WITH FORCING AXIOMS}\label{Relationships}

In this section, we will prove Theorems \ref{MPvsFA}, \ref{NothingImpliesAnything},\ref{MM++vsMP}, and \ref{MM++vsStar}.

\begin{proof}[Proof of Theorem \ref{MPvsFA}]

Assume $\rm{MP}$$_{\Sigma_1}(\omega_1,\Gamma)$ and let $\mathbb{P}\in \Gamma$. We show that 
\begin{center}
$(H_{\omega_2},\in)\prec_{\Sigma_1}(H_{\omega_2}^{V^\mathbb{P}},\in)$.
\end{center}
Suppose $y\in H_{\omega_2}$, $\psi(v_0,v_1)$ is a $\Sigma_0$ formula. Clearly, if $(H_{\omega_2},\in)\vDash \exists x \psi[x,y]$, then  $(H_{\omega_2}^{V^\mathbb{P}},\in)\vDash \exists x \psi[x,y]$. So suppose $(H_{\omega_2}^{V^\mathbb{P}},\in)\vDash \exists x \psi[x,y]$. Then for any $\mathbb{Q}\in \Gamma^{V^\mathbb{P}}$, we have
\begin{center}
$(H_{\omega_2}^{V^{\mathbb{P}\ast \mathbb{Q}}},\in)\vDash \exists x \psi[x,y]$.
\end{center}
This means the $\Sigma_1$ statement ``$\exists x \psi[x,y]$" is $\Gamma$-necessary over $V^\mathbb{P}$ so it holds in $V$ by MP$_{\Sigma_1}(\omega_1,\Gamma)$. We easily get in $V$, $(H_{\omega_2},\in)\vDash \exists x \psi[x,y]$.

Conversely, suppose $\sf{FA}_{\aleph_1}(\Gamma)$ holds. Let $\phi(v)$ be $\Sigma_1$ and $A\subseteq \omega_1$ be such that there is some $\mathbb{P}\in\Gamma$ such that $\phi(A)$ is $\Gamma$-necessary over $V^{\mathbb{P}}$. In particular, $\phi(A)$ is true in $V^{\mathbb{P}}$. Working in $V^{\mathbb{P}}$, by reflection, there is some $\kappa$ such that $(H_\kappa,\in)\vDash \phi(A)$. Let $M$ be the transitive collapse of the Skolem hull of $A$ in $(H_\kappa,\in)$ (we may fix a well-ordering of $H_\kappa$ and define our Skolem functions relative to this well-ordering). Since $|M|=\aleph_1$, $M$ is transitive, and $\phi$ is $\Sigma_1$, $(H^{V^\mathbb{P}}_{\omega_2},\in)\vDash \phi(A)$. By $\sf{FA}_{\aleph_1}(\Gamma)$, in $V$, $(H_{\omega_2},\in)\vDash \phi(A)$. 
\end{proof}

\begin{proof}[Proof of Theorem \ref{NothingImpliesAnything}]
\indent (1) We start with a model $M$ of $\rm{MP}(\omega_1,\Gamma)$. Working in $M$, let $\mathbb{P}$ be the $< \omega_2$-strategically closed forcing that adds a $\square_{\omega_1}$-sequence (cf. \cite[Example 6.6]{cummings2010iterated}). Since $\mathbb{P}$ does not add new $\omega_1$-sequences of ordinals, it's easy to see that $M^\mathbb{P}\vDash \rm{MP}$$(\omega_1,\Gamma)$. On the other hand, the existence of a $\square_{\omega_1}$ sequence implies $\sf{MM}^{++}$ fails in $M^\mathbb{P}$.

\indent (2) The first clause follows from Theorem \ref{MPvsFA}. We prove the second clause as follows. By \cite[Theorem 10.99]{Woodin}, whenever $M\vDash \sf{AD}^+$ and $M$ is closed under $\M_1^\sharp$, that is for any $a\in M$, $\M_1^\sharp(a)\in M$ (and is fully iterable in $M$), then for any $G\subseteq \mathbb{P}_{\rm{max}}$ generic over $M$, $M[G] \vDash \sf{BMM}$. Fix such an $M$ (we may assume $M\vDash \sf{AD}^+ + \neg \sf{AD}_{\mathbb{R}}$) and let 
\begin{center}
$N = L^{\M_1^\sharp}(\powerset(\mathbb{R})^M)$.\footnote{Let $F$ be the operator $x\mapsto \M_1^\sharp(x)$. As usual, we put $J_0(\powerset(\mathbb{R})^M) = tr. cl. (\powerset(\mathbb{R})^M)$. We take union at limit steps and $J^{\M_1^\sharp}_{\alpha+1}(\powerset(\mathbb{R})^M) = F(J^{\M_1^\sharp}_\alpha(\powerset(\mathbb{R})^M))$. See \cite{CMI} for more details on the fine structure and the exact stratification of this kind of hierarchies.}
\end{center}

It is well-known that the $\M_1^\sharp$ operator relativizes well, that is for any $a,b$, if $a\in L(b)$ then $\M_1^\sharp(a)\in L(\M_1^\sharp(b))$. It follows that $N$ is closed under $\M_1^\sharp$. By the definition of $N$, $N$ is not closed under $\M_2^\sharp$.

Now let $G\subseteq \mathbb{P}_{\rm{max}}$ be $N$-generic. Then by \cite[Theorem 10.99]{Woodin}, $N[G]\vDash \sf{ZFC} + \sf{BMM}$. By \cite[Theorem 4.49]{Woodin}, in $N[G]$, the ideal NS$_{\omega_1}$ is saturated. Furthermore, since $N$ is not closed under $\M_2^\sharp$, $N[G]$ isn't either. 

We claim that MP$(\omega_1,\Gamma)$ cannot hold in $N[G]$. Suppose not. Then by the proof of Theorem \ref{gettingPD}, $N[G]$ is closed under $\M_n^\sharp$ for all $n$. This contradicts the fact that $N[G]$ is not closed under $\M_2^\sharp$. 

\indent (3) The second clause follows from the fact that MP$(0,\Gamma)$ is consistent relative to $\sf{ZFC}$ and MP$(\omega,\Gamma)$ implies $0^\sharp$ exists (cf. Theorem \ref{BMM}). 

For the first clause, start with a model $M$ of MP$(\omega,\Gamma)$. In $M$, let  $\mathbb{P}$ be the poset that adds a Cohen subset of $\omega_1$ (conditions in $\mathbb{P}$ are countable functions from $\omega_1$ into $2$ ordered by end-extensions). Since $\mathbb{P}$ does not add new countable sequences of elements of $M$, we get that $M^\mathbb{P}\vDash \rm{MP}(\omega,\Gamma)$. Furthermore, $M^\mathbb{P}\vDash 2^{\aleph_0}=\aleph_1$, so by Theorem \ref{BMM}, MP$(\omega_1,\Gamma)$ must fail in $M^\mathbb{P}$.
\end{proof}

The proof of Theorem \ref{MM++vsMP} uses a characterization of $\sf{MM}^{++}$ from \cite{viale2011martin}.
\begin{lemma}
\label{MM++equivalence}
Suppose there are class many Woodin cardinals. Then the following are equivalent:
\begin{enumerate}[(1)]
\item $\sf{MM}^{++}$.
\item For every Woodin cardinal $\delta$ and every stationary set preserving forcing $\mathbb{P}\in V_\delta$, there is a complete embedding $i:\mathbb{P}\rightarrow \mathbb{B}=_{\rm{def}}\rm{r.o.}$$(\mathbb{P}_{<\delta}\rest T)$\footnote{$\mathbb{P}_{<\delta}\rest T$ is just $\mathbb{P}_{<\delta}$ restricted to conditions below $T$, where $\mathbb{P}_{<\delta}$ is the full stationary tower at $\delta$. Also, the notion of complete embedding is defined as in \cite{viale2011martin}.} for some stationary set $T\in V_\delta$ such that 
\begin{center}
$T\Vdash_{\mathbb{P}_{<\delta}} \rm{crt}(\dot{j}) > \omega_1^V \wedge |\dot{\mathbb{P}}| = \omega_1 \ $\footnote{$\dot{j}$ is a canonical $\mathbb{P}_{<\delta}$-name for the generic embedding induced by a $\mathbb{P}_{<\delta}$-generic.} 
\end{center}
and 
\begin{center}
$\Vdash_{\mathbb{P}} \mathbb{B}\slash i[\mathbb{P}]$ is stationary set preserving.
\end{center}
\end{enumerate}
\end{lemma}

\begin{proof}[Proof of Theorem \ref{MM++vsMP}]

For the item 1., we prove MP$_{\Pi_2}(\omega_1,\Gamma)$ holds. Let $\phi$ be a $\Pi_2$ formula, $x$ be a subset of $\omega_1$, and $\mathbb{P}$ be a staionary set preserving forcing such that in $V^{\mathbb{P}}$, $\phi [x]$ is necessary with respect to further stationary set preserving forcings. We will show that $\phi [x]$ holds in $V$. 

Suppose not. Then $\phi [x]$ fails in $V$ and there exists some inaccessible $\gamma$ such that $\phi [x]$ fails also in $V_{\gamma}$. Let $\delta $ be a Woodin cardinal such that $\text{rank}(\mathbb{P}), \gamma  < \delta$. By the characterization of $\text{MM}^{++}$ under the presence of proper class many Woodin cardinals by Viale~\cite{viale2011martin} (cf. Lemma \ref{MM++equivalence}), there exists a $\mathbb{P}_{<\delta}$-generic $H$ over $V$ and $\mathbb{P}$-generic $G$ over $V$ in $V[H]$ such that $\mathbb{P}_{<\delta}/G$ is a stationary set preserving forcing in $V[G]$. Since $\phi [x]$ is necessary with respect to further stationary set preserving forcings in $V^{\mathbb{P}}$, in particular $\phi [x]$ holds in $V[H]$. 

We will derive a contradiction by showing that $\phi [x] $ actually {\it fails} in $V[H]$. Since $H$ is a $P_{<\delta}$-generic over $V$, there is a generic embedding $j \colon V \to M \subseteq V[H]$ associated to $H$ such that the critical point of $j$ is $\omega_2^V$, $j (\delta) = \delta$, and $V_{\delta}^M = V_{\delta}^{V[H]}$. Since $\phi [x]$ fails in $V_{\gamma}$ and $\gamma$ is inaccessible in $V$, by elementarity of $j$, $\phi [x]$ fails also in $V_{j(\gamma)}^M$ and $j(\gamma)$ is inaccessible in $M$. Since $\gamma < \delta$ and $j(\delta) = \delta$, $j(\gamma) < \delta$. Also since $V_{\delta}^M = V_{\delta}^{V[H]}$, $\phi [x]$ fails in $V_{j(\gamma)}^{V[H]}$ and $j(\gamma)$ is inaccessible in $V[H]$. Since $j(\gamma)$ is inaccessible in $V[H]$, $\Sigma_2$ statements are upward absolute from $V_{j(\gamma)}^{V[H]}$ to $V[H]$. In particular, the negation of $\phi [x]$ holds in $V[H]$ and so $\phi [x]$ fails in $V[H]$. Contradiction! This finishes the proof of the item 1..

For the item 2., assume that there are proper class many Woodin cardinals and a supercompact cardinal. By the arguments for Theorem~66 in \cite{MR3304634}, there is a class forcing extension where $V = \text{gHOD}$ holds and there are proper class many Woodoin cardinals and a supercompact cardinal\footnote{gHOD is the intersection of HODs of all set generic extensions of $V$, which is generically invariant, i.e. $\text{gHOD}^V = \text{gHOD}^{V^{\mathbb{P}}}$ for any set forcing poset $\mathbb{P}$.}. Let $\kappa$ be the least supercompact cardinal in this class forcing extension. Over this class forcing extension, we further force $\text{MM}^{++}$ using the supercompactness of $\kappa$. Let us call this further forcing extension as $W$. 

We shall show that $\text{MP}(\omega_1, \Gamma)$ {\it fails} in $W$. In $W$, $\kappa$ is equal to $\omega_2$ while $\kappa$ is the least supercompact in gHOD of $W$. Consider the statement \lq\lq the least supercompact cardinal in gHOD is less than $\omega_2$". This statement is easily seen to be necessary after collpsing $\omega_2$ in $W$ by a stationary set preserving forcing over $W$. However, this statement is false in $W$ since $\kappa = \omega_2^W$ is the least supercompact in gHOD. So $\text{MP}(\omega_1, \Gamma)$ fails in $W$, as desired. This finishes the proof of the item 2..

For the item 3., we assume that there are a supercompact cardinal $\kappa$ and boundedly many (but at least one) inaccessibles above $\kappa$. We force $\text{MM}^{++}$ using the supercompact $\kappa$. Let us call this forcing extension $W$ and we shall show that $\text{MP}(\omega_1, \Gamma)$ {\it fails} in $W$. In $W$, there are boundedly many (but at least one) inaccessibles. So there will be no inaacessible after collapsing a sufficiently big cardinal using a stationary set preserving forcing in $W$. Therefore, the $\Pi_2$ statement \lq\lq There is no inaccessible cardinal" is necessary in a stationary set preserving forcing extension of $W$. However, this statement is {\it false} in $W$ because there is at least one inaccesible in $W$. Therefore, $\text{MP}(\omega_1, \Gamma)$ for $\Pi_2$ statements is false in $W$, as desired. This finises the proof of the item 3..

\end{proof}

\begin{proof}[Proof of Theorem \ref{MM++vsStar}]
The item 1. follows from the fact that Woodin's Axiom (*) holds in the model $N[G]$ in the proof of Theorem~\ref{NothingImpliesAnything} because $N[G]$ is a $\mathbb{P}_{\text{max}}$ forcing extenion of a model of $\sf{AD}^+$ while $\text{MP}(\omega_1, \Gamma)$ fails in $N[G]$ as discussed in the proof of Theorem~\ref{NothingImpliesAnything}.

The item 2. can be verified by looking at the construction of a model in Thoerem~7.1 in \cite{MR2474445} where $\sf{MM}$ holds while Woodin's Axiom (*) fails . Larson actually showed that there is a lightface definable wellordering of $H_{\omega_2}$ over $H_{\omega_2}$ which implies the failure of Woodin's Axiom (*) by the homogenity of $\mathbb{P}_{\text{max}}$ forcing. Starting with a model of set theory where there are proper class many strongly compact cardinals and an inaccessible cardinal which is also reflecting, one can modify the construction of Larson's model in such a way that there is a lightface definable wellordering of $H_{\omega_2}$ over $H_{\omega_2}$ whereas $\text{MP}(\omega_1, \Gamma)$ (instead of $\sf{MM}$) holds  by the same argument as the one for Theorem~\ref{FromSupercompact}.

\end{proof}

\section{LOWER-BOUND CONSISTENCY STRENGTH}\label{Lowerbounds}

In this section, we prove Theorems \ref{BMM} and \ref{gettingPD}.

\begin{proof}[Proof of Theorem \ref{BMM}]

By Theorem \ref{MPvsFA}, $\rm{MP}(\omega_1,\Gamma)$ implies $\sf{BMM}$. By \cite{todorcevic2002generic}, $\sf{BMM}$ implies
\begin{center}
$2^{\aleph_0}=2^{\aleph_1}=\aleph_2$.
\end{center} 
By \cite{schindler2004semi}, $\sf{BMM}$ implies ``for all $X$, $X^\sharp$ exists". This gives the first clause. 

In fact, the proof of \cite{schindler2004semi} gives that MP$(\omega,\Gamma)$ implies for all $X$, $X^\sharp$ exists. We sketch the argument here. Suppose MP$(\omega,\Gamma)$ holds but for some $X$, $X^\sharp$ does not exist. By \cite{CodingIntoK}, there is a stationary-set preserving set forcing extension $V[g]$ of $V$ in which there is some $a\subseteq \omega$ such that $X\in H_{\omega_2}=L_{\omega_2}[a]$ (the argument uses the non-existence of $X^\sharp$ to construct such an $a$). In $V[g]$, the following statement holds:
\begin{center}
$\exists a\exists M (a\in M\cap \mathbb{R} \wedge |M|=\omega_1 \wedge \omega_1\in M \wedge M \textrm{ is transitive } \wedge M\vDash``\sf{ZFC}$$^- + a \textrm{ codes a reshaped subset of }\omega_1$".\footnote{See \cite[Definition 3.2]{schindler2004semi} for the definition of reshaped subsets.}
\end{center} 

The above is a $\Sigma_1$-statement $\Gamma$-necessary over $V[g]$, hence by MP$(\omega,\Gamma)$, it is true in $V$. Hence there are $a, M\in V$ satisfying the above statement. Then $a$ really does code a reshaped subset of $\omega_1$ in $V$ by the absoluteness of the coding. By \cite[Lemma 3.3]{schindler2004semi}, there is a stationary set preserving extension $V[h]$ in which there is $b\subseteq \omega$ coding a reshaped subset of $\omega_1$ and $f_b <^* f_a$ \footnote{The function $f_a$ is defined relative to $a$ as in \cite[Definition 3.2]{schindler2004semi}; similarly for $f_b$. $f_b<^* f_a$ if there is a club $C\subseteq \omega_1$ such that for all $\alpha\in C$, $f_b(\alpha) < f_a(\alpha)$.}. In $V[h]$, the following $\Sigma_1$ statement $\phi(a)$ is true:
\begin{center}
$\exists b\exists M (b\in M\cap \mathbb{R} \wedge |M|=\omega_1 \wedge \omega_1\in M \wedge M \textrm{ is transitive } \wedge M\vDash``\sf{ZFC}$$^- + b \textrm{ codes a reshaped subset of }\omega_1 + f_b <^* f_a$".
\end{center}

We get that $\phi(a)$ is $\Gamma$-necessary over $V[h]$ and hence by MP$(\omega,\Gamma)$, in $V$,  there is a real $b$ that codes a reshaped subset of $\omega_1$ and $f_b <^* f_a$. Repeat the above argument ad infinitum, we get that $<^*$ is ill-founded, which is absurd. This completes the proof of the second clause.

For the last clause, let $M$ be a model of ``$V = L(\wp(\mathbb{R})) + \textsf{AD}_\mathbb{R} + \Theta$ is regular." Let $G\subseteq \mathbb{P}_{\rm{max}}^M$ be $M$-generic and $H\subseteq \textrm{Coll}(\Theta,\wp(\mathbb{R}))^{M[G]}$ be $M[G]$-generic. Then $M[G][H] \vDash ``\textsf{ZFC} + \sf{MM(c)}$$ + V = L[X]$ for some $X\subseteq \omega_3"$. Fix such an $X$. If $M[G][H]=L[X] \vDash \rm{MP}(\omega_1,\Gamma)$, then $L[X] \vDash ``X^\sharp$ exists". Contradiction. Hence $\sf{MM(c)}$ does not imply $\rm{MP}(\omega_1,\Gamma)$.
\end{proof}

Now we prove Theorem \ref{gettingPD}. We need the following fact, whose proof follows from \cite[Theorem 0.3]{claverie2012woodin}.
\begin{theorem}\label{weakCovering}
Suppose $\mathcal{I}$ is a precipitous ideal on $\kappa$. Suppose for some $n$, $V$ is closed under the $\M_n^\sharp$-operator (if $n=0$, we write $\M_0^\sharp(x)$ for $x^\sharp$). Suppose there is no inner model with $n+1$ Woodin cardinals. Then the core model $K$ exists and $\kappa^+ = (\kappa^+)^K$.
\end{theorem}
\begin{proof}[Proof of Theorem \ref{gettingPD}]
Let $\mathcal{I}$ be a precipitous ideal on $\omega_1$. By induction on $n$, we prove:
\begin{enumerate}[$(1)_n$]
\item $H_{\omega_1}$ is closed under $\M_n^\sharp$.
\item $H_{\omega_2}$ is closed under $\M_n^\sharp$.
\item $V$ is closed under $\M_n^\sharp$.
\end{enumerate}
\noindent \underline{The base case:} Assume $n=0$. If $x\in H_{\omega_1}$ then $\M_0^\sharp(x)$ exists by the fact that there is a precipitous ideal on $\omega_1$. So $(1)_0$ holds. To see $(2)_0$, it suffices to show that for every $A\subseteq \omega_1$, $\M_0^\sharp(A)$ exists. Let $G$ be $V$-generic for the forcing $\wp(\omega_1)\slash \mathcal{I}$. Let $i: V\rightarrow M \subseteq V[G]$ be the generic embedding induced by $G$. Then in $M$, $A\in H_{\omega_1}$, so $M\vDash \M_0^\sharp(A)$ exists. So $A$ does exist (in $V[G]$ and in $V$) since $M$ contains all the ordinals. 

Now we verify $(3)_0$. Suppose not. Let $\nu > \omega_2$ be such that there is some $A\in H_\nu$ such that $\M_0^\sharp(A)$ doesn't exist. Let $\mathbb{P} = \rm{Coll}(\omega_1,A)$. Note that $\mathbb{P}\in \Gamma$. Then the statement ``there is some $A\in H_{\omega_2}$ such that $\M_0^\sharp(A)$ doesn't exist" is $\Gamma$-necessary over $V^\mathbb{P}$, hence true in $V$. But this is a contradiction to $(2)_0$.
\\
\noindent \underline{The successor step:} Suppose $(3)_n$ holds for some $n$. We show $(1)_{n+1}-(3)_{n+1}$.  At this point, we have the following analogue of \cite[Lemma 7.5]{claverie2012woodin}, the proof of which is exactly that of \cite[Lemma 7.5]{claverie2012woodin}. As a matter of notation, for a mouse $\P$ over some $x$, we write $\P \leq \M_{n}^\sharp$ if either $\P$ is $n$-small or $\P = \M_n^\sharp(x)$.

All the following lemmas make use of our assumptions: MP$(\omega_1,\Gamma)$ and $\mathcal{I}$ is a precipitous ideal on $\omega_1$. We will not state them in the lemmas' statements.

\begin{lemma}
\label{absoluteness}
Suppose $(3)_n$ holds. Let $G\subseteq \wp(\omega_1)\slash \mathcal{I}$ be $V$-generic. Let $i:V\rightarrow M\subseteq V[G]$ be the generic embedding induced by $G$. Let $x\in M$ be a set of ordinals. Suppose $\P\in M$ be such that $\M\vDash ``\P$ is a mouse such that $\P\leq \M_n^\sharp$". Then $V[G] \vDash \P$ is a mouse. 
\end{lemma}
We should mention that one of the uses of Lemma \ref{absoluteness} is in proving Theorem \ref{weakCovering}. The point is that the proof of Theorem \ref{weakCovering} requires that $j(K)$ is iterable in $V[G]$ and Lemma \ref{absoluteness} guarantees this. We now show $(1)_{n+1}$. 
\begin{lemma}
\label{Mn+1sharpForReals}
Assume $(3)_n$ holds. Let $x\in \mathbb{R}$. Then $\M_{n+1}^\sharp(x)$ exists.   
\end{lemma}
\begin{proof}
For notational simplicity, suppose $x=\emptyset$. Suppose $\M_{n+1}^\sharp$ doesn't exist. Then $K^c$ does not have a Woodin cardinal, $(n+1)$-small, and is fully iterable via the $\Q$-structure guided strategy (the $\Q$-structures are given by the $\M_n^\sharp$ operator, which is defined on all of $V$ by $(3)_n$). This means $K$ exists and is fully iterable.

Let $i,G,M$ be as in Lemma \ref{absoluteness}. Let $\kappa=\omega_1^V$. By Theorem \ref{weakCovering}, $(\kappa^+)^K = \kappa^+$. Let $A\subseteq \kappa$ code $K||\kappa$ in $V$. Let $\mathbb{P} = \textrm{Coll}(\omega_1,\omega_2)$. Let $H\subseteq \mathbb{P}$ be $V$-generic. So $K^V = K^{V[G]}$ and 
\begin{center}
$(\kappa^+)^K=\omega_2^V < (\kappa^+)^{V[G]} = \omega_2^{V[G]}$.
\end{center}
By local definability of $K$, the displayed statement can be expressed in $V[G]$ by a formula $\phi(A)$. In fact, $\phi(A)$ is $\Gamma$-necessary over $V^\mathbb{P}$ and hence is true in $V$. This contradicts the fact that $(\kappa^+)^K = \omega_2^V$.
\end{proof}
Now we show $(2)_{n+1}$.
\begin{lemma}\label{Mn+1sharpForHomega2}
Suppose $(3)_n$ and $(1)_{n+1}$ holds. Then $(2)_{n+1}$ holds.
\end{lemma}
\begin{proof}
Let $A\subseteq \omega_1$. Let $i,G,M$ be as in Lemma \ref{absoluteness}. $A\in M$ and $A\in H_{\omega_1}^M$. By elementarily, $(1)_{n+1}$ holds in $M$ and hence $\M_{n+1}^\sharp(A)$ exists in $M$. Since $(3)_n$ holds, Lemma \ref{absoluteness} implies $(\M_{n+1}^\sharp(A))^M = (\M_{n+1}^\sharp(A))^{V[G]}$. This is because $\M_{n+1}^\sharp(A)$ is iterable by the $\Q$-structure guided strategy and these $\Q$-structures are mice $\leq \M_n^\sharp$. 

Write $\P$ for $(\M_{n+1}^\sharp(A))^M$. Let $H\subseteq \textrm{Coll}(\omega, \omega_3^V)$ be $V[G]$ generic. Then $\P$ is in fact fully iterable in $V[G][H]$ via a the following strategy $\Sigma$: $\Sigma(\T) = $ the unique $b$ such that $\Q(b,\T)$ exists and $\Q(b,\T)\unlhd \M_{n}^\sharp(\mathcal{M}(\T))$ (this uses the basic fact that $V$ is closed under $\M_n^\sharp$ implies $V[G]$ is closed under $\M_n^\sharp$ for any generic extension $V[G]$ of $V$). 

It's clear that there is a $J\subseteq \textrm{Coll}(\omega,\omega_3^V)$ be $V$-generic such that $V[J]=V[G][H]$. By homogeneity, $\P\in V$ and is iterable in $V$ by $\Sigma$ (i.e. $\Sigma\cap V \in V$). This proves the lemma.
\end{proof}
Finally, we verify $(3)_{n+1}$.
\begin{lemma}\label{Mn+1sharpArbitrary}
Suppose $(3)_n$ and $(2)_{n+1}$ hold. Then $(3)_{n+1}$ holds.
\end{lemma}
\begin{proof}
Suppose there is some $\nu>\omega_2$ and some set of ordinals $A\in H_{\nu}$ such that $\M_{n+1}^\sharp(A)$ doesn't exist. Let $\mathbb{P} = \textrm{Coll}(\omega_1,A)$. Again, $\mathbb{P}\in\Gamma$. Let $G\subseteq \mathbb{P}$ be $V$-generic. By arguments above, the statement ``there is some $A\in H_{\omega_2}$ such that $\M_{n+1}^\sharp(A)$ doesn't exist" is $\Gamma$-necessary over $V^\mathbb{P}$ and hence is true in $V$. This contradicts our assumption.
\end{proof}
This completes the proof of the theorem.
\end{proof}

\begin{remark}
A careful examination of the proof of the theorem shows that one does not need all of MP$(\omega_1,\Gamma)$. One simply need MP$(0,\Gamma)$.
\end{remark}

\section{OPEN PROBLEMS AND QUESTIONS}\label{OpenProblems}

In this section, we list some questions and open problems related to results proved above.

\begin{question}
Is $\rm{NMP}$$(\omega_1,\Gamma)$ consistent where $\Gamma$ is the class of $\sigma$-closed forcings?
\end{question}

We believe the question can be answered positively in light of Woodin's (unpublished) proof that NMP is consistent. 

\begin{question}
What is the exact consistency strength of MP$(\omega_1,\Gamma)$ and of MP$(\omega,\Gamma)$ for $\Gamma$ being the class of stationary-set preserving forcings?
\end{question}

We conjecture the following in light of Theorem \ref{gettingPD}.
\begin{conjecture}
Assume the hypothesis of Theorem \ref{gettingPD}. Then there is a model of $\omega$ Woodin cardinals.
\end{conjecture}

In \cite{SchindlerMMStar}, the author defines the axiom $\sf{MM}^{*,++}$ which implies both $\sf{MM}^{++}$ and Woodin's axiom (*). It is natural to ask whether $\sf{MM}^{*,++}$ implies MP$(\omega_1,\Gamma)$ for $\Gamma$ being the class of stationary-set preserving forcings.
\begin{question}
Let $\Gamma$ be the class of stationary-set preserving forcings. Does $\sf{MM}^{*,++}$ imply $\rm{MP}$$(\omega_1,\Gamma)$?
\end{question}

Let $\Gamma$ be as in Theorem \ref{MM++vsMP}. The formula produced in Theorem \ref{MM++vsMP}(2) is more complicated than $\Sigma_2$. By Theorem \ref{MM++vsMP}(1), MP$_{\Pi_2}(\omega_1,\Gamma)$ follows from $\sf{MM}$$^{++} + $ there is a proper class of Woodin cardinals.

\begin{question}
Let $\Gamma$ be as in Theorem \ref{MM++vsMP}. Assume $\sf{MM}$$^{++} + $ there is a proper class of Woodin cardinals. Must \rm{MP}$_{\Sigma_2}(\omega_1,\Gamma)$ hold?
\end{question}

\bibliographystyle{plain}
\small
\bibliography{NMP}
\end{document}